\documentclass[a4paper,11pt]{amsproc}

\usepackage{xargs}
\usepackage{mathtools, todonotes}
\usepackage{amsmath,amssymb,amsfonts,amsthm}
\usepackage{thmtools}

\definecolor{citegreen}{rgb}{0,0.5,0.15}
\definecolor{linkblue}{rgb}{0,0.4,1}
\definecolor{citebordergreen}{rgb}{0,0.9,0.37}
\usepackage[colorlinks,linkbordercolor={0 0 1}, linkcolor=blue, unicode=true, citecolor=citegreen, linkbordercolor={0 0.4 1}, citebordercolor={0 0.9 0.37}, backref]{hyperref}

\usepackage[numeric]{amsrefs}
\usepackage{enumerate, tikz-cd, nicefrac}

\usepackage{breqn}
\usepackage{mathrsfs}

\title{The anabelian restricted Burnside problem}

\author{Andreas Thom}
\address{Andreas Thom, TU Dresden, 01062 Dresden, Germany}
\email{andreas.thom@tu-dresden.de}

\theoremstyle{plain}
\newtheorem{theorem}{Theorem}[section]

\newtheorem{proposition}[theorem]{Proposition}
\newtheorem{lemma}[theorem]{Lemma}
\newtheorem{corollary}[theorem]{Corollary}

\theoremstyle{definition}

\newcommand{\beq}{\begin{equation}}
\newcommand{\eeq}{\end{equation}}
\newcommand{\beqn}{\begin{equation*}}
\newcommand{\eeqn}{\end{equation*}}
\newcommand{\brq}{\begin{dmath}[compact]}
\newcommand{\erq}{\end{dmath}}
\newcommand{\brqn}{\begin{dmath*}[compact]}
\newcommand{\erqn}{\end{dmath*}}
\newcommand{\bag}{\begin{align}}
\newcommand{\eag}{\end{align}}
\newcommand{\bagn}{\begin{align*}}
\newcommand{\eagn}{\end{align*}}

\newcommand{\vertiii}[1]{{|\kern-0.2ex|\kern-0.2ex| #1 
    |\kern-0.2ex|\kern-0.2ex|}}

\begin{document}
\begin{abstract}
Let $n,d \in \mathbb N$ and $w \in \mathbb F_n$ be non-trivial. We prove that the relatively free group of rank $d$ in the variety defined by the group law $w$ has a largest anabelian finite quotient and estimate its size. Here, a finite group is called anabelian if it has only non-abelian composition factors. The estimate is based on explicit bounds for the length of laws for finite simple groups obtained by Bradford and the author and on recent work by Fumagalli--Leinen--Puglisi.
\end{abstract}

\maketitle

\tableofcontents

\section{Introduction}

Let $n \in \mathbb N$ and $w \in \mathbb F_n$ be a non-trivial word. For $d \in \mathbb N$, we denote the relatively free group of rank $d$ in the variety defined by a group law $w$ by $B(d,w)$.

In the case $w=x^{k}$ one writes $B(d,k):=B(d,x^k)$ for the free Burnside group of rank $d$ and exponent $k$, see \cite{Burnside1902} for the original reference. For small exponents, finiteness is classical: $B(d,2)$ is elementary abelian; $B(d,3)$ is finite with $|B(d,3)|=3^{d+\binom d2+\binom d3}$ and nilpotent of class at most $3$ \cite{LeviVdW1933}; $B(d,4)$ and $B(d,6)$ are finite, see \cites{Hall1958, Sanov1940}. The landscape changes for larger $k$. Novikov--Adyan \cite{NovikovAdyan1968, Adian1979} proved that for all sufficiently large odd exponents $k\ge 665$ one has $B(d,k)$ infinite for every $d\ge2$. For even exponents, later work established infiniteness for all sufficiently large even $k$; early breakthroughs were due to Ivanov \cite{Ivanov1994}, and Lys\"enok \cite{Lysenok1996} obtained general results covering large even exponents. For small exponents outside $\{2,3,4,6\}$ the situation is still largely open.

The restricted Burnside problem asks: for given $d,k$, is there a largest finite $d$-generated group of exponent $k$? Equivalently, is the restricted Burnside group $R(d,k) := B(d,k)/K_{d,k} $ finite, where $K_{d,k}$ is the intersection of all subgroups of finite index of $B(d,k)$? This was solved affirmatively in a series of works combining group theory with Lie methods. Kostrikin \cite{Kostrikin1959} proved finiteness for prime exponents. Hall--Higman \cite{HallHigman1956} developed a reduction showing that it suffices to treat prime powers. Zel'manov \cites{Zelmanov1991odd, Zelmanov1992two} completed the solution for all $k$ by establishing powerful theorems on Lie algebras with polynomial identities; in particular, $R(d,k)$ is finite for every $d,k$. While this guarantees the existence of a largest finite quotient of $B(d,k)$, effective general bounds on $|R(d,k)|$ are enormous; sharper estimates are known in several small-prime cases via $p$-group and Lie-theoretic techniques. 

\vspace{0.2cm}

If the total exponent of $w$ in the variables is zero, then $\mathbb Z^d$ is naturally a quotient of $B(d,w)$. Similar phenomena arise for words deeper in the central series of the free group, when the free nilpotent group arises as a natural infinite residually finite quotient. De Cornulier--Mann \cite{MR2395788} proved among other things that for $n=2, w=[x,y]^{30}$, the group $R(d,w)$ is not virtually solvable, where we define $R(d,w)$ in a natural way as the largest residually finite quotient of $B(d,w).$ In particular, in many cases there is no bound on the size of finite quotients of $B(d,w)$.
In this case, i.e.\ $w \in [\mathbb F_n,\mathbb F_n]$, the group $B(d,w)$ has been studied in detail for various interesting examples, e.g.\ \cite{AdianAtabekyan2018, Golod1969, Juhasz1990}.

If $w \not\in [\mathbb F_n,\mathbb F_n]$, it seems much harder to understand when $B(d,w)$, which is then a quotient of a classical Burnside group, is infinite, and we are not aware of any positive results apart from the classical case. In any case, $R(d,w)$ is always finite under this assumption as a consequence of Zel'manov's theorem.

\vspace{0.2cm}

The main result of this paper is a general bound on the size of finite quotients, once one restricts the class in a natural way.
Recall, that every finite group $G$ admits a composition series
$$\{1\} =G_0 \lhd G_1 \lhd \cdots \lhd G_k = G $$ for some $k \in \mathbb N$ such that the subquotients $G_i/G_{i-1}$ are simple for all $1 \leq i \leq k.$ These factor groups are called composition factors. The classical Jordan-H\"older theorem says that the multiset of isomorphism classes of composition factors does not depend on the composition series. Following Nikolov and Segal, we say that a finite group is \emph{anabelian} if all composition factors are non-abelian. As a consequence of rather deep results, anabelian groups behave in various ways like finite simple groups, see for example \cite{MR3557468}.

In order to state our main result, let us denote by ${\rm E}_{n}$ the $n$-th iterated exponential function, i.e. ${\rm E}_0(x)=x$ and ${\rm E}_{n+1}(x) = \exp({\rm E}_n(x)).$

\begin{theorem} \label{thm:main}
Let $d,n \in \mathbb N$ and $w \in \mathbb F_n \setminus \{1\}$ of length $\ell$.
Any finite anabelian quotient of $B(d,w)$ is of size 
at most ${\rm E}_{2\ell}(O(d \ell \ln(\ell)))$.
\end{theorem}

Towers of exponentials are not unexpected in questions of this kind. Indeed, Vaughan-Lee and Zelmanov proved bounds of similar form for the classical restricted Burnside problem. One result in this context says that a finite quotient of $B(d,k)$ is of size at most ${\rm E}_{m}(d)$ with $m=k^{k^k}, $ see \cites{MR1421888, MR1717419}.

The proof of our main result relies on the estimates for the length of laws for finite simple groups obtained by Bradford and the author in \cite{BradfordThom2024} and the seminal work of Fumagalli--Leinen--Puglisi \cite{MR4508338}.

\section{Nonsolvable length and anabelian groups} \label{sec:nons}

In order to summarize the results from \cite{MR4508338}, we need another definition: Every finite group $G$ has a normal series each of whose factors is
either a solvable group or a direct product of non-abelian simple groups. The
minimum number of nonsolvable factors, attained on all possible such series, is called the \emph{nonsolvable length} $\lambda(G)$ of the group $G$.
The main result of \cite{MR4508338} is the following theorem:

\begin{theorem}[Fumagalli--Leinen--Puglisi] \label{thm:puglisi}
If a finite group $G$ satisfies a law of length $\ell$, then its nonsolvable length $\lambda(G)$ is bounded by $\ell$.
\end{theorem}

While solvable factors do not come up in our analysis, the result is already highly non-trivial for anabelian groups. The authors of \cite{MR4508338} mention a conjecture attributed to  Larsen, saying that $\lambda(G)$ could be bounded by $O(\ln(\ell))$. A positive result in this direction would have immediate consequences and imply near optimal bounds in Theorem \ref{thm:main}.

\vspace{0.2cm}

Let us collect some basic properties of the class of anabelian groups.

\begin{lemma} \label{lem:anab}
The class of anabelian groups is closed under extensions, quotients and subdirect products.
\end{lemma}
\begin{proof} For a subdirect product, note that a composition series maps onto a normal series of each of the factors, with subquotients either simple or trivial. Hence, if there is a abelian composition factor in the subdirect product, not all  of the factors in the product can be anabelian. \end{proof}

\begin{corollary}
Let $\Gamma$ be a countable group. There exists a unique quotient $ \pi \colon \Gamma \to \Gamma_{\rm an}$ such that a surjection $\varphi \colon \Gamma \to H$ onto a finite group factors through $\pi$ if and only if $H$ is anabelian and such homomorphism separate elements of $\Gamma$.
\end{corollary}
\begin{proof}
Consider the partially ordered set $\mathcal P$ of finite index subgroups $\Lambda$ of $\Gamma$, such that $\Gamma/\Lambda$ is anabelian. By the previous lemma, $\mathcal P$ is closed under finite intersections and hence directed. We set $\Lambda_0 = \cap_{\Lambda \in \mathcal P} \Lambda$ and define $\Gamma_{\rm an} = \Gamma/\Lambda_0.$
\end{proof}

It is natural to define $R_{\rm an}(d,w):= B(d,w)_{\rm an} = R(d,w)_{\rm an}$. We call $R_{\rm an}(d,w)$ the anabelian restricted Burnside group -- by its definition, it is the largest residually finite-anabelian quotient of $B(d,w)$. However, we can record the following corollary of our main theorem.

\begin{corollary}
Let $n,d \in \mathbb N$, and $w \in \mathbb F_n \setminus \{1\}$. Then, $R_{\rm an}(d,w)$ is finite.
\end{corollary}

The Fumagalli--Leinen--Puglisi structure result about finite groups satisfying a law has direct implications on the structure of $R(d,w)$ as well, but we were not able to say anything more specific.

\section{Length of laws and applications}

By the classification of finite simple groups (CFSG), see \cite{GLS-I, AS-I, AS-II}, a non-abelian finite simple group is either an alternating group ${\rm Alt}_{k}$ for $k \geq 5$, a finite simple group of Lie type or one of 26 sporadic groups. The finite simple groups of Lie type come in families, indexed by a Lie type $X$, one symbol from the list $A_k$, $^{2}A_k$, $B_k$, $C_k$, $D_k$, $^{2}D_k$, $E_6$, $^{2}E_6$, $E_7$, $E_8$, $F_4$, $G_2$, $^{3}D_4$, $^{2}B_2$, $^{2}F_4$, or $^{2}G_2$. Each finite simple group of Lie type is then of the form $X_k(q)$, where $q$ is a prime power, which is restricted to powers of $2$ in the cases ${}^{2}B_2(q), ^{2}F_4(q)$ and powers of $3$ in the case $^{2}G_2(q)$. We define $a=a(X)\in\mathbb N$ as in the following table, where some values depend on divisibility properties of $k$ and $q$:

\begin{table}[h]
\centering
\renewcommand{\arraystretch}{1.2}
\begin{tabular}{c|cccccc}
$X$ & $A_k$ & ${}^{2}A_k$ & $B_k$ & $C_k$ & $D_k$ & ${}^{2}D_k$  \\
\hline
$a$ & $\lfloor(k+1)/2\rfloor$ & $\lfloor(k+1)/2\rfloor$ &
$2\lfloor k/2\rfloor \mbox{ or } k $ &
$k$ &
$k-1 \mbox{ or } k-2$ &
$2\lfloor k/2\rfloor$ 
\end{tabular}

\begin{tabular}{c|cccccccccc}
$X$ & $E_6$ & ${}^{2}E_6$ & $E_7$ & $E_8$ & $F_4$ & $G_2$ & ${}^{3}D_4$ & ${}^{2}B_2$ & ${}^{2}F_4$ & ${}^{2}G_2$\\
\hline
$a$ & $4$ & $4$ & $7$ & $7$ & $4$ & $1$ & $3$ & $1$ & $2$ & $1$
\end{tabular}

\label{tab:a-values}
\end{table}

The main result of \cite{BradfordThom2024} gives matching lower and upper bounds (up to poly-logarithmic factors and apart from the case ${}^{2}B_2$) for the length of shortest laws in $G$. Let us only record the lower bounds, since these are the ones we will need:

\begin{theorem}\label{thm:BT-lower}
For $G=X(q)$ as above, every nontrivial word $w\in\mathbb F_2$ that is a law for $G$ has length at least $\Omega(q^{a})$, except when $G={}^{2}B_2(q)$, in which case every law has length at least $\Omega(q^{1/2})$.
\end{theorem}

There is no comparable result for alternating groups, where the best known lower bound for the length of a law for ${\rm Alt}_k$ is linear in $k$. Together with the results from \cite{BradfordThom2024}, this refines a more classical result of Jones \cite{MR344342} saying that there are only finitely many non-isomorphic finite simple groups satisfying a law of given length.
We will see that the bound for ${\rm Alt}_k$, which is likely far from the truth (see \cite{MR3937332}*{Problem 4.3}), is the bottleneck for our estimates. See also \cite{kozmathom2016divisibility} for quasi-polynomial upper bounds on the length of laws.

As a key step, we need to prove a bound on the size of a $d$-generated product of non-abelian finite simple groups that satisfy a law of length $\ell$. The following lemma is straightforward, see also \cite{MR1681530}*{Lemma 5}.

\begin{lemma} \label{lem:product}
Let $G_1,\dots,G_t$ be pairwise nonisomorphic finite simple groups and consider $G= G_1^{\times n_1} \times \cdots \times G_t^{\times n_t}$ for $n_1,\dots,n_t \in \mathbb N$. Then, $G$ is $d$-generated if and only if $G_m^{n_m}$ is $d$-generated for each $1 \leq m \leq t.$
\end{lemma}

\begin{lemma} \label{lem:alternating}
Let $G$ be a $d$-generated product of simple alternating groups that satisfies a law of length $\ell$. Then, $|G| \leq {\rm E}_2(O(d \ell \ln(\ell))).$
\end{lemma}
\begin{proof}
We may assume that $G = \prod_{m=5}^t {\rm Alt}^{\times n_m}_m$ with $n_m \in \mathbb N$. Since $G$ satisfies a law of length $\ell$, we may assume that $t \leq \ell.$
Moreover, if ${\rm Alt}_m^{n_m}$ is $d$-generated, then $n_m \leq |{\rm Alt}_m|^d$ and in particular $n_m \leq \exp(d m \ln(m)).$ We conclude that $$|G| \leq \exp\left( \sum_{m=5}^\ell m \ln(m) \exp(d m \ln(m)) \right) \leq \exp(\ell^2 \ln(\ell) \exp(d \ell \ln(\ell))) .$$ This implies the claim.
\end{proof}

Better bounds hold for finite simple groups of Lie type. Here, we need a rough upper bound for the size of the finite simple group $X(q)$. Indeed, we have $|X(q)| \leq O(q^{b(X)})$ for $b(X)$ in the following table:

\begin{table}[h]
\centering

\begin{tabular}{c|cccccccc}
$X$&$A_k(q)$ & ${}^{2}A_k(q^{2})$ & $B_k(q)$ & $C_k(q)$ & $D_k(q)$ & ${}^{2}D_k(q^{2})$ &${}^{2}B_2(q)$ & ${}^{3}D_4(q^{3})$\\
\hline
$b(X)$ &$k^{2}+2k$ & $k^{2}+2k$ & $2k^{2}+k$ & $2k^{2}+k$ & $2k^{2}-k$ & $2k^{2}-k$ &$5$ & ${28}$\\
\end{tabular}

\begin{tabular}{c|cccccccc}
$X$&  $F_4(q)$ & ${}^{2}F_4(q)$ & $G_2(q)$ & ${}^{2}G_2(q)$ & $E_6(q)$ & ${}^{2}E_6(q^{2})$ & $E_7(q)$ & $E_8(q)$\\
\hline
$b(X)$& ${52}$ & ${26}$ & ${14}$ & ${7}$ & ${78}$ & ${78}$ & ${133}$ & ${248}$\\
\end{tabular}
\end{table}

The following result is the analogue of Lemma \ref{lem:alternating} for simple groups of a fixed Lie type.

\begin{lemma} \label{lem:Lie}
Let $G$ be a $d$-generated product of finite simple groups of Lie type $X$ that satisfies a law of length $\ell$. Then, $|G| \leq {\rm E}_2(O(d \ln(\ell )^2)).$
\end{lemma}
\begin{proof} The proof follows the ideas of the proof of Lemma \ref{lem:alternating}. Let's start with one of the infinite families. We need to estimate a product
$$\prod_{q,k} |X_k(q)|^{|X_k(q)|^d},$$
where the product runs over all $q,k$, such that $X_k(q)$ satisfies a law of length $\ell$. Roughly speaking, this implies that $q^{c_1k} \leq \ell$ for some constant $c_1$, while $|X_k(q)| \leq q^{c_2 k^2}$. In particular, $k \leq c_2 \ln(\ell)$ and $q \leq \ell^{1/(c_1k)}.$ Thus,
$$|G| \leq \prod_{k=1}^{c_2 \ln(\ell)}\prod_{q=1}^{\ell^{1/(c_1k)}}|X_k(q)|^{|X_k(q)|^d} \leq \exp\left(c_3 \ln(\ell)^3 \exp(  \ell^{c_4 d \ln(\ell)}) \right),$$
where $c_2,c_3$ and $c_4$ are constants.
Thus we obtain an upper bound of the form
${\rm E}_2\left(O(d \ln(\ell)^2)\right)$ as claimed. The same bound applies for families of bounded rank.
\end{proof}

Since, using the CFSG, there are only finitely many Lie types and finitely many sporadic groups, we obtain from Lemma \ref{lem:product}, Lemma \ref{lem:alternating} and Lemma \ref{lem:Lie} the following proposition:

\begin{proposition} \label{prop:bound}
Let $d,n \in \mathbb N$ and $w \in \mathbb F_n$. Let $G$ be a $d$-generated product of non-abelian finite simple groups that satisfies a law of length $\ell$, then
$$|G| \leq {\rm E}_2(O(d\ell \ln(\ell))).$$
\end{proposition}

We are now ready for the proof of the main theorem.

\section{Proof of the main theorem}

Theorem \ref{thm:main} follows  from a straightforward induction argument on the size of $d$-generated finite groups of nonsolvable length $k$ satisfying a law of length $\ell.$ In order to streamline the estimates, note that we have that ${\rm E}_k$ is monotone, ${\rm E}_k(x) \leq {\rm E}_l(x)$ for $k \leq l$ and ${\rm E}_k(x){\rm E}_k(y) \leq {\rm E}_k(xy)$ for all $k,l \in \mathbb N$, $x,y \geq 2.$

\begin{proposition} \label{prop:nonsolvable}
The size of a $d$-generated finite anabelian group of nonsolvable length $k$ satisfying a law of length $\ell$ is bounded by
${\rm E}_{2k}(O(d \ell \ln(\ell))).$
\end{proposition}
\begin{proof}
Let $c$ be the constant implicit in Proposition \ref{prop:bound}. We assume without loss of generality that $c,\ell \geq 2$, so that $x:= c d \ell \ln(\ell) \geq 2.$ We intend to prove an upper bound of ${\rm E}_{2k}(2x)$ by induction on $k$. The claim for $k=1$ was proven in the previous section. If $G$ is a $d$-generated finite anabelian group of nonsolvable length $k$, then it fits into a short exact sequence
$$1 \to H \to G \to G/H \to 1,$$
where $H$ is anabelian of nonsolvable length $k-1$ and $G/H$ is $d$-generated product of finite non-abelian simple groups. If $G$ satisfies a law of length $\ell$, then so does $H$ and $G/H$. We conclude that the size of $G/H$ is bounded by ${\rm E}_2(x)$ by Proposition \ref{prop:bound}. Hence, by Schreier's theorem, the group $H$ is generated by $(d-1){\rm E}_2(x) +1 \leq d{\rm E}_2(x) $ elements. Assuming by induction that the size of $H$ is bounded by
${\rm E}_{2(k-1)}(2 {\rm E}_2(x) x),$ we obtain
\begin{eqnarray*}
|G| &\leq& {\rm E}_{2(k-1)}(2 {\rm E}_2(x) x) \cdot {\rm E}_2(x) \\
&\leq& {\rm E}_{2(k-1)}(2 {\rm E}_2(x) x) \cdot {\rm E}_{2(k-1)}(x) \\
&\leq& {\rm E}_{2(k-1)}\left(2 {\rm E}_2(x)x^2  \right) \\
&\leq& {\rm E}_{2k}\left(x + \ln \ln(2x^2)  \right) \\
&\leq&{\rm E}_{2k}\left(2x \right).
\end{eqnarray*}
This finishes the proof.
\end{proof}

Now, the previous proposition readily implies the main theorem using Theorem \ref{thm:puglisi}. This finishes the proof of the main theorem.

\section{Examples}
\label{sec:examples}

The case $w \not \in [\mathbb F_n,\mathbb F_n]$ is called the \emph{periodic} case; indeed, by a result of B.H.\ Neumann \cite{BHNeumann1937}, satisfying a law $w \not \in [\mathbb F_n,\mathbb F_n]$ is equivalent to satisfying laws $x^n,w_0 \in [\mathbb F_n,\mathbb F_n]$. Here, $n$ is the ${\rm gcd}$ of the total orders of the variables appearing in $w$. In this case $B(d,w)$ is a quotient of $B(d,n)$. Thus, if $n$ is odd, it follows that every finite quotient of $B(d,w)$ is of odd order and hence solvable by the Feit--Thompson theorem. The same conclusion holds if $n$ is of the form $n=2^ap^b$ for $a,b \in \mathbb N$ and a prime $p$ by Burnside's theorem. In particular, $R_{\rm an}(d,w)$ is trivial in those cases. On the other side, if $n=2\cdot 3 \cdot 5=30$ and $d \geq 2$, then $R_{\rm an}(d,30)$ is non-trivial since ${\rm Alt}_5$ satisfies the law $x^{30}$.

\vspace{0.2cm}

There is another case that can be understood directly:

\begin{proposition}
Let $w_i \in \mathbb  F_{n_i}$ for $1 \leq i \leq 2$ and $w=[w_1,w_2] \in \mathbb F_{n}$ for $n=n_1+n_2$ be the commutator in disjoint variables. Then, $R_{\rm an}(d,w)$ naturally embeds into $R_{\rm an}(d,w_1) \times R_{\rm an}(d,w_2)$ as a subdirect product.
\end{proposition}
\begin{proof}
Consider the verbal subgroup of $R_{\rm an}(d,w)$ associated with $w_i$ by $N_i$. It is easy to see that $[N_1,N_2]$ is trivial since $R_{\rm an}(d,w)$ satisfies the law $w=[w_1,w_2]$, and hence $N_1 \cap N_2$ abelian. Since $R_{\rm an}(d,w)$ is anabelian, it follows that the normal subgroup $N_1 \cap N_2$ is trivial. Hence, $R_{\rm an}(d,w)$ embeds into $R_{\rm an}(d,w)/N_1 \times R_{\rm an}(d,w)/N_2$. However, $R_{\rm an}(d,w)/N_i$ is anabelian by Lemma \ref{lem:anab} and satisfies the law $w_i$. We conclude that also the natural homomorphism $R_{\rm an}(d,w) \to R_{\rm an}(d,w_1) \times R_{\rm an}(d,w_2)$ is injective. This finishes the proof.
\end{proof}

The previous proposition covers the example $w=[x^n,y]$ that has been studied by Adian--Atabekyan \cite{AdianAtabekyan2018} and we get that $R_{\rm an}(d,w)$ is isomorphic to $R_{\rm an}(d,n)$ in this case.

\vspace{0.2cm}

To contrast the finiteness of $R_{\rm an}(d,w)$, it would be very interesting to have more results saying that $B(d,w)$ is infinite (or maybe even non-amenable) for suitable words $w$. It seems natural to conjecture that if $w$ is sufficiently complicated, then $B(d,w)$ is infinite.

\section*{Acknowledgments}

The author thanks Henry Bradford and Jakob Schneider for helpful discussions, and Goulnara Arzhantseva for carefully reading of an early draft, for discussions related to Section \ref{sec:examples}, and for drawing attention to references \cite{AdianAtabekyan2018} and \cite{BHNeumann1937}.


\begin{bibdiv}
\begin{biblist}

\bib{Adian1979}{book}{
  author={Adian, Sergei I.},
  title={The Burnside Problem and Identities in Groups},
  series={Ergebnisse der Mathematik und ihrer Grenzgebiete},
  volume={95},
  publisher={Springer},
  date={1979},
}

\bib{AdianAtabekyan2018}{article}{
  author={Adian, Sergei I.},
  author={Atabekyan, Varujan S.},
  title={Central extensions of free periodic groups},
  journal={Mat. Sb.},
  volume={209},
  date={2018},
  number={12},
  pages={3–16},
  translation={
    journal={Sb. Math.},
    volume={209},
    date={2018},
    number={12},
    pages={1677–1689}
  },
}

\bib{AS-I}{book}{
  author={Aschbacher, Michael},
  author={Smith, Stephen D.},
  title={The classification of quasithin groups. I. Structure of strongly quasithin $K$-groups},
  series={Mathematical Surveys and Monographs},
  volume={111},
  publisher={American Mathematical Society},
  place={Providence, RI},
  date={2004}
}
\bib{AS-II}{book}{
  author={Aschbacher, Michael},
  author={Smith, Stephen D.},
  title={The classification of quasithin groups. II. Main theorems: the classification of simple QTKE-groups},
  series={Mathematical Surveys and Monographs},
  volume={112},
  publisher={American Mathematical Society},
  place={Providence, RI},
  date={2004}
}

\bib{MR3937332}{article}{
   author={Bradford, Henry},
   author={Thom, Andreas},
   title={Short laws for finite groups and residual finiteness growth},
   journal={Trans. Amer. Math. Soc.},
   volume={371},
   date={2019},
   number={9},
   pages={6447--6462},
}

\bib{BradfordThom2024}{article}{
  author={Bradford, Henry},
  author={Thom, Andreas},
  title={Short laws for finite groups of {L}ie type},
  journal={J. Eur. Math. Soc.},
  date={2024},
  status={published online first},
  eprint={arXiv:1811.05401},
}

\bib{Burnside1902}{article}{
  author={Burnside, William},
  title={On an unsettled question in the theory of discontinuous groups},
  journal={Quart. J. Pure Appl. Math.},
  volume={33},
  date={1902},
  pages={230--238},
}

\bib{MR2395788}{article}{
   author={de Cornulier, Yves},
   author={Mann, Avinoam},
   title={Some residually finite groups satisfying laws},
   conference={
      title={Geometric group theory},
   },
   book={
      series={Trends Math.},
      publisher={Birkh\"{a}user, Basel},
   },
   date={2007},
   pages={45--50},
}

\bib{MR4508338}{article}{
   author={Fumagalli, Francesco},
   author={Leinen, Felix},
   author={Puglisi, Orazio},
   title={An upper bound for the nonsolvable length of a finite group in
   terms of its shortest law},
   journal={Proc. Lond. Math. Soc. (3)},
   volume={125},
   date={2022},
   number={5},
   pages={1066--1082},
}

\bib{Golod1969}{article}{
  author={Golod, Evgeny S.},
  title={Some problems of Burnside type},
  journal={Amer. Math. Soc. Transl. Ser.~2},
  volume={84},
  date={1969},
  pages={83--88}
}

\bib{GLS-I}{book}{
  author={Gorenstein, Daniel},
  author={Lyons, Richard},
  author={Solomon, Ronald},
  title={The classification of the finite simple groups},
  series={Mathematical Surveys and Monographs},
  volume={40},
  number={1–10},
  publisher={American Mathematical Society},
  place={Providence, RI},
  date={1994–2023},
  note={With Inna Capdeboscq as coauthor for Nos.\ 9–10}
}

\bib{Hall1958}{article}{
  author={Hall, Marshall, Jr.},
  title={Solution of the Burnside problem for exponent $6$},
  journal={Illinois J. Math.},
  volume={2},
  date={1958},
  pages={764--786},
}

\bib{HallHigman1956}{article}{
  author={Hall, Philip},
  author={Higman, Graham},
  title={On the $p$-length of $p$-soluble groups and reduction theorems for Burnside's problem},
  journal={Proc. London Math. Soc. (3)},
  volume={6},
  date={1956},
  pages={1--42},
}

\bib{Ivanov1994}{article}{
  author={Ivanov, Sergei V.},
  title={The free Burnside groups of sufficiently large exponents},
  journal={Int. J. Algebra Comput.},
  volume={4},
  number={1--2},
  date={1994},
  pages={1--308},
}

\bib{MR1681530}{article}{
   author={Jarden, Moshe},
   author={Lubotzky, Alexander},
   title={Random normal subgroups of free profinite groups},
   journal={J. Group Theory},
   volume={2},
   date={1999},
   number={2},
   pages={213--224},
}

\bib{MR344342}{article}{
   author={Jones, Gareth A.},
   title={Varieties and simple groups},
   journal={J. Austral. Math. Soc.},
   volume={17},
   date={1974},
   pages={163--173},
}

\bib{Juhasz1990}{article}{
  author={Juhász, Arye},
  title={Engel groups I},
  journal={Israel J. Math.},
  volume={69},
  date={1990},
  pages={1--24}
}

\bib{Kostrikin1959}{article}{
  author={Kostrikin, Alexei I.},
  title={The Burnside problem},
  journal={Izv. Akad. Nauk SSSR Ser. Mat.},
  volume={23},
  date={1959},
  pages={3--34},
}

\bib{kozmathom2016divisibility}{article}{
author={Kozma, Gady},
author={Thom, Andreas},
title={Divisibility and laws in finite simple groups},
journal={Mathematische Annalen},
volume={364},
date={2016},
number={1-2},
pages={79--95},
}

\bib{LeviVdW1933}{article}{
  author={Levi, Friedrich},
  author={van der Waerden, Bartel Leendert},
  title={\"Uber eine besondere Klasse von Gruppen},
  journal={Abh. Math. Sem. Univ. Hamburg},
  volume={9},
  date={1933},
  pages={154--158},
}

\bib{Lysenok1996}{article}{
  author={Lys\"enok, Igor G.},
  title={Infinite Burnside groups of even exponent},
  journal={Izv. Ross. Akad. Nauk Ser. Mat.},
  volume={60},
  number={3},
  date={1996},
  pages={3--224},
  note={English transl.: \emph{Izv. Math.} 60 (1996), 485--517},
}

\bib{BHNeumann1937}{article}{
  author={Neumann, Bernhard H.},
  title={Identical relations in groups. I},
  journal={Mathematische Annalen},
  volume={114},
  date={1937},
  pages={506–525}
}

\bib{MR3557468}{article}{
   author={Nikolov, Nikolay},
   title={Verbal width in anabelian groups},
   journal={Israel J. Math.},
   volume={216},
   date={2016},
   number={2},
   pages={847--876},
}

\bib{NovikovAdyan1968}{article}{
  author={Novikov, Pyotr S.},
  author={Adyan, Sergei I.},
  title={Infinite periodic groups I, II, III},
  journal={Izv. Akad. Nauk SSSR Ser. Mat.},
  volume={32},
  date={1968},
}

\bib{Sanov1940}{article}{
  author={Sanov, Ivan N.},
  title={Solution of Burnside's problem for exponent $4$},
  journal={Uchen. Zap. Leningrad Gos. Univ. Ser. Mat.},
  volume={10},
  date={1940},
  language={Russian},
}

\bib{MR1421888}{article}{
   author={Vaughan-Lee, Michael},
   author={Zel'manov, Efim I.},
   title={Upper bounds in the restricted Burnside problem. II},
   journal={Internat. J. Algebra Comput.},
   volume={6},
   date={1996},
   number={6},
   pages={735--744},
}

\bib{MR1717419}{article}{
   author={Vaughan-Lee, Michael},
   author={Zel'manov, Efim I.},
   title={Bounds in the restricted Burnside problem},
   note={Group theory},
   journal={J. Austral. Math. Soc. Ser. A},
   volume={67},
   date={1999},
   number={2},
   pages={261--271},
}

\bib{Zelmanov1991odd}{article}{
  author={Zel'manov, Efim I.},
  title={Solution of the restricted Burnside problem for groups of odd exponent},
  journal={Math. USSR-Izvestiya},
  volume={36},
  number={1},
  date={1991},
  pages={41--60},
}

\bib{Zelmanov1992two}{article}{
  author={Zel'manov, Efim I.},
  title={A solution of the restricted Burnside problem for $2$-groups},
  journal={Math. USSR-Sbornik},
  volume={72},
  number={2},
  date={1992},
  pages={543--565},
}


\end{biblist}
\end{bibdiv}
\end{document}